\documentclass[12pt]{article}
\usepackage{amsfonts,amsmath,amsthm, amssymb}
\usepackage{latexsym, euscript, epic, eepic}
\newtheorem{cro}{Corollary}[section]
\newtheorem{defn}{Definition}[section]
\newtheorem{prop}{Proposition}[section]
\newtheorem{thm}{Theorem}[section]
\newtheorem{lem}{Lemma}[section]

\begin{document}
 \title{  On the topological pressure of the saturated set with non-uniform structure
\footnotetext {
*Corresponding author\\
2010 Mathematics Subject Classification: 37D25, 37D35
  }}
\author{Cao Zhao$^{1},$ Ercai Chen$^{1,2_{*}}.$ \\
   \small   1 School of Mathematical Science, Nanjing Normal University,\\
    \small   Nanjing 210023, Jiangsu, P.R. China\\
     \small 2 Center of Nonlinear Science, Nanjing University,\\
         \small   Nanjing 210093, Jiangsu, P.R. China\\
          \small    e-mail: izhaocao@126.com \\
          \small    e-mail: ecchen@njnu.edu.cn \\
           }

\date{}
\maketitle

\begin{center}
 \begin{minipage}{120mm}
{\small {\bf Abstract.}  We derive a conditional variational principle of the saturated set for systems with the non-uniform structure. Our result applies to a broad class of systems including $\beta$-shifts, $S$-gap shifts and their factors. }
\end{minipage}
 \end{center}
\section{ Introduction.}

Most results in multifractal analysis  are applied to research the local asymptotic quantities, such as Birkhoff averages, Lyapunov exponents, local entropies, and pointwise dimensions, which reveal information about a single point or trajectory. It is of interest  to study the level set for these quanties.
For a topological dynamical system $(X,d, \sigma)$(or $(X,\sigma)$ for short) consisting of a compact metric space $(X,d)$ and a continuous map $\sigma:X\to X$.
For a continuous function $\psi:X\to\mathbb R,$ we always consider the following set
$$X(\psi,\alpha)=\bigg\{x\in X:\lim\limits_{n\to\infty}\frac{1}{n}\sum\limits_{i=0}^{n-1}\psi(\sigma^ix)=\alpha\bigg\}.$$
 The level set $X(\psi,\alpha)$ is the multifractal decomposition set of  ergodic averages of $\psi$. There are fruitful results about the descriptions of the structure (Hausdorff
dimension or topological entropy  or topological pressure) of these
level sets  in topological dynamical systems.
We refer the reader to
\cite{BarSau,BarSauSch,Ols,OlsWin,PeiChe,PfiSul,TakVer,Tho2,ZhoChe} and the references therein.   In \cite{PfiSul},  Pfister and Sullivan  consider the saturated set and obtained a conditional variational principle. Let $C(X )$ be the space of continuous functions from $X$ to $\mathbb R.$ For $\varphi\in C(X)$ and $n\geq 1$, denote $S_{n}\varphi(x):=\sum_{i=0}^{n-1}\varphi(\sigma^{i}x)$. Denote by $M(X )$ , $M_{\sigma}(X)$ and $M^{e}_{\sigma}(X)$ the set of Borel probability measures on $X $, the collection of all $\sigma$-invariant Borel probability measures and all $\sigma$-ergodic invariant Borel probability measures, respectively. It is well-known that $M(X)$ and $M_{\sigma}(X)$ equipped with weak$^*$ topology are compact metrizable spaces. 
There exists a countable and separating set of continuous functions $\{f_{1}, f_{2}\cdots \}$ with $0\leq f_{i}(x)\leq 1$ on $X$ such that
$$D(\mu,\nu):=||\mu-\nu||=\sum_{k\geq 1}\frac{|\int f_{k}d\mu-\int f_{k}d\nu|}{2^{k}},$$
 defines a metric for the weak$^{*}$ topology on $M_{\sigma}(X)$.
Denote  the limit point set of $\{x_{n}\}_{n\geq 1}$ by $A(x_{n})$.
Define
$$\mathcal{E}_{n}(x):=\dfrac{\sum_{i=0}^{n-1}\delta_{\sigma^{i}x}}{n}.$$
The generic set for $\mu\in M_{\sigma}(X)$ can be denoted by
$$G_{\mu}(X,\sigma):=\{x\in X:~\mathcal{E}_{n}(x)\to \mu\}.$$
For any compact connected subset $K\subset M_{\sigma}(X)$, define the saturated set for $K$ as follows,
$$G_{K}(X,\sigma):=\{x\in X:~A(\mathcal{E}_{n}(x))=K\}.$$
We define the multifractal spectrum for $\psi$ to be
 $\mathcal{L}_{\psi}:=\{\alpha\in \mathbb R: X(\psi, \alpha)\neq \emptyset\}.$
   In \cite{PfiSul}, Pfister and Sullivan  showed the following theorem.
   \begin{thm}
   If $(X,\sigma)$ satisfies the $g$-almost product property and the uniform separation property, then for any compact connected non-empty set $K\subset M_{\sigma}(X)$,
   $$\inf \{h_{\mu}(\sigma):~\mu\in K\}=h_{top}(  G_{K}(X,\sigma)),$$
   where $h_{top}(\cdot)$ denotes the Bowen topological entropy.
   \end{thm}
 We can see that the above theorem need the conditions  $g$-almost product property  and the uniform separation property.  While,  $g$-almost product property condition is a kind of specification property which holds for the case of $\beta$-shifts. For $S$-gap shifts, and the factors of $\beta$-shifts and $S$-gap shifts, these are the first results on the saturated set. In this paper, we consider a class of symbolic systems which was studied in \cite{CliTho}. That is, $(X,\sigma)$ is a symbolic system with non-uniform structure which is mainly defined as there exist  $\mathcal{G}\subset \mathcal{L}(X)$ has $(W)$-specification and  $\mathcal{L}(X)$ is edit approachable by $\mathcal{G}$.  The detail of definitions will be given in the next section. This can be considered as another kind of specification property which holds for both $\beta$-shifts, $S$-gap shifts and their factors.

Our main results is the following.

\begin{thm}\label{main}
Let $X$ be a shift space with $\mathcal{L}=\mathcal{L}(X)$ and $\varphi :X\to \mathbb R $ be a continuous function. Suppose that $\mathcal{G}\subset \mathcal{L}$ has $(W)$-specification and  $\mathcal{L}$ is edit approachable by $\mathcal{G}$, then for any compact connected subset $K\subset M_{\sigma}(X)$. We have
$$P_{G_{K}(X,\sigma)}(\varphi)=\inf_{\mu\in K}\big\{h_{\mu}(\sigma)+\int\varphi d\mu\big\},$$
where $P_{\bullet}(\varphi)$ denote topological pressure.
\end{thm}
Accordingly, we give the result about irregular set.
Set
$$ \hat{X}(\psi )  :=\{x\in X:~\lim\limits_{n\to\infty}\frac{S_{n}\psi(x)}{n}~\text{does not exist}\}.$$
 \begin{thm}\label{main2}
 Let $X$ be a shift space with $\mathcal{L}=\mathcal{L}(X)$ and $\varphi :X\to \mathbb R $ be a continuous function. Suppose that $\mathcal{G}\subset \mathcal{L}$ has $(W)$-specification and  $\mathcal{L}$ is edit approachable by $\mathcal{G}$, then for any $\psi\in C(X)$,  either $\hat{X}(\psi )=\emptyset$, or
$$P_{\hat{X}(\psi )}(\varphi)= P_{X}(\varphi) ,$$
where $P_{\bullet}(\varphi)$ denote topological pressure.
\end{thm}
This paper is organized as follows: In section 2, we give our definitions and some key propositions. In section 3, we give the proof of main results. In section 3, we give some applications to the $\beta$-shifts and $S$-gap shifts and their factors.

\section{Preliminaries}
In this paper, we consider the symbolic space. Let $p\geq 2$ be an integer and $\mathcal{A}=\{1,\cdots,p\}$. Define
$$\mathcal{A}^{\mathbb N}=\{(w_{i})_{i=1}^{\infty}:~w_{i}\in \mathcal{A}~\text{for} ~i\geq 1\},$$
and then $\mathcal{A}^{\mathbb N}$ is compact endowed with the product discrete topology.
And we can define the metric of $\mathcal{A}^{\mathbb N}$ as follows, for any $u,v\in \mathcal{A}^{\mathbb N}$, define
\begin{align*}
d(u,v):= e^{-|u\wedge v|},
\end{align*}
where $|u\wedge v|$ denote the maximal length $n$ such that $u_{1}=v_{1}, u_{2}=v_{2},\cdots, u_{n}=v_{n}$ and $d(u,v)=0$ if $u=v.$
 We say that $(X,\sigma)$ is a subshift over $\mathcal{A}$ if $X$ is a compact subset of $\mathcal{A}^{\mathbb N}$, and $\sigma(X)\subset X$, where $\sigma$ is the left shift map on $\mathcal{A}^{\mathbb N}$ defined by
$$\sigma((w_{i})_{i=1}^{\infty})=(w_{i+1})_{i=1}^{\infty},~~\forall~~ (w_{i})_{i=1}^{\infty}\in \mathcal{A}^{\mathbb N}.$$
In particular, $(X,\sigma)$ is called the full shift over $\mathcal{A}$ if $X=\mathcal{A}^{\mathbb N}.$ For $n\in \mathbb N$ and $w\in \mathcal{A}^{\mathbb N}$, we write
$$[w]=\{(w_{i})_{i=1}^{\infty}\in \mathcal{A}^{\mathbb N}: ~~~w_{1}\cdots w_{n}=w\}$$
and call it an $n$-th word in $\mathcal{A}^{\mathbb N}.$ The language of $X$, denoted by $\mathcal{L}=\mathcal{L}(X)$, is the set of finite words that appear in some $x\in X$. More precisely,
$$\mathcal{L}(X)=\{w\in \mathcal{A}^{*}: [w]\cap X\neq \emptyset\},$$
where $\mathcal{A}^{*}=\cup_{n\geq 0}\mathcal{A}^{n}$   is the set of sequences $x\in X$ that begin with the word $w$. Given $w\in \mathcal{L}$, let $|w|$ denote the length of $w$. For any collection $\mathcal{D}\subset \mathcal{L}$, let $\mathcal{D}_{n}$ denote $\{w\in \mathcal{D}: |w|=n\}.$ Thus, $\mathcal{L}_{n}$ is the set of all words of length $n$ that appear in sequences belonging to $X$. Given words $u,v$ we use juxtaposition $uv$ to denote the word obtained by concatenation.
 
We give the topological pressure for non-compact set in symbolic space.
\begin{defn}
Let $X$ be a subshift space on a finite alphabet and $Z\subset X$ be arbitrary Borel set. For $N \in \mathbb N $, and $t\in \mathbb R$. We define the following quantities:
 $$M(Z,t,\varphi, N)=\inf\Big\{\sum_{[w_{0}w_{1}\cdots w_{m}]\in \mathcal{S}}\exp\Big(-t(m+1)+\sup_{w\in[w_{0}w_{1}\cdots w_{m}] }\sum_{k=0}^{m}\varphi(\sigma^{k}x)\Big)\Big\},$$
where the infimum is taken over all  finite or countable collections $\mathcal{S} $ of cylinder sets $[w_{0}w_{1}\cdots w_{m}]$ with $m\geq N $ which cover $Z.$
Define
$$M(Z,t,\varphi )=\lim\limits_{N\to\infty}M(Z,t,\varphi, N).$$
The existence of the limit is guaranteed since the function $M(Z,t,\varphi, N)$ does not decrease with $N$.
By standard techniques, we can show the existence of
$$P_{Z}(\varphi )=\inf\{t:~M(Z,t,\varphi )=0\}.$$
And then, we define the topological pressure of $Z$ by $P_{Z}(\varphi )$.
 If $\varphi=0$, then $P_{Z}(0)=h_{top}(Z),$ where $h_{top}(Z)$ denotes the Bowen topological entropy.
\end{defn}

\begin{defn}\rm{\cite{CliTho}}
Given a shift space $X$ and its language $\mathcal{L}$, consider a subset $\mathcal{G}\subset \mathcal{L}$. Given $\tau\in \mathbb N$, we say that $\mathcal{G}$ has $(W)$-specification with gap length $\tau$ if for every $v, w\in \mathcal{G}$ there is $u\in \mathcal{L}$ such that $vuw\in \mathcal{G}$ and $|u|\leq \tau.$

\end{defn}
\begin{defn}\rm{\cite{CliTho}}
Define an edit of a word $w=w_{1}\cdots w_{n}\in \mathcal{L}$ to be a transformation of $w$ by one of the following actions, where $w^{j}\in \mathcal{L}$ are arbitrary words and $a,a^{'}\in \mathcal{A}$ are arbitrary symbols.

(1) substitution: $w=u^{1}au^{2}\mapsto w^{'}=u^{1}a^{'}u^{2}.$

(2) Insertion: $w=u^{1}u^{2}\mapsto w^{'}=u^{1}a^{'}u^{2}.$

(3) Deletion: $w=u^{1}au^{2}\mapsto w^{'}=u^{1}u^{2}.$
\end{defn}
\noindent Given $v,w\in \mathcal{L}$, define the edit distance between $v$ and $w$ to be the minimum number of edits required to transform the word $v$ into the word $w$: we will denote this by $\hat{d}(v,w).$
The following lemma about describe the  size of balls in the edit metric.
\begin{prop}\rm{\cite{CliTho}}\label{prop2.1}
There is $C>0$ such that given $n\in \mathbb N, w\in \mathcal{L}_{n}$, and $\delta>0$, we have
$$\sharp\{v\in \mathcal{L}:~\hat{d}(v,w)\leq \delta n\}\leq Cn^{C}(e^{C\delta}e^{-\delta\log\delta})^{n}.$$
\end{prop}
\noindent Next we introduce the key definition, which requires that any word in $\mathcal{L}$ can be transformed into a word in $\mathcal{G}$ with a relatively small number of edits.

\begin{defn}\rm{\cite{CliTho}}
Say that a non-decreasing function $g:\mathbb N\to \mathbb N$ is a mistake function if $\frac{g(n)}{n}$ converges to $0$. We say that $\mathcal{L}$ is edit approachable by $\mathcal{G}$, where $\mathcal{G}\subset \mathcal{L}$, if there is a mistake function $g$ such that for every $w\in \mathcal{L}$, there exists $v\in \mathcal{G}$ with $\hat{d}(v,w)\leq g(|w|)$.
\end{defn}
\begin{lem}\rm{\cite{CliTho}}\label{lem2.1}
For any continuous function $\varphi\in C(X)$ and any mistake function $g(n):\mathbb N\to \mathbb N$, there is a sequence of positive numbers $\delta_{n}\to 0$ such that if $x,y\in X$ and $m,n\in\mathbb N$ are such that $\hat{d}(x_{1}\cdots x_{n},y_{1}\cdots y_{m})\leq g(n),$ then $|\frac{1}{n}S_{n}\varphi(x)-\frac{1}{m}S_{m}\varphi(y)|\leq \delta_{n}.$
\end{lem}
\noindent Similar to the above lemma, we can give another lemma for measure.
\begin{lem}\label{lem2.2}
For any mistake function $g(n)$, there is a sequence of positive numbers $\delta_{n}\to 0$ such that if $x,y\in X$ and $m,n\in\mathbb N$ are such that $\hat{d}(x_{1}\cdots x_{n},y_{1}\cdots y_{m})\leq g(n),$ then $D(\mathcal{E}_{n}(x), \mathcal{E}_{m} (y))\leq \delta_{n}.$
\end{lem}
\begin{proof}
 It is easy to get, by a slight modification  of the proof of Lemma \ref{lem2.1} in \cite{CliTho}.
\end{proof}
\noindent  We can get the following lemma, by applying[\cite{CliTho}, Proposition 4.2 and Lemma 4.3].
\begin{prop}\rm{\cite{CliTho}}\label{prop2.2}
If $\mathcal{G}$ has $(W)$-specification, then there exist $\mathcal{F}\subset \mathcal{L}$, which has free concatenation property (if for all $u,w\in \mathcal{F}$, we have $uw\in\mathcal{F}$) such that $\mathcal{L}$ is edit approachable by $\mathcal{F}.$
\end{prop}

 \section{Proof of Theorem \ref{main}}
   The upper bound is easy to get by  Theorem 3.1 in \cite{PeiChe}.
 Firstly, we begin with a proposition about measure.
\begin{prop}\label{prop3.1}
Let $\alpha_{i}\geq 0  $    $\beta_{i}\geq 0$ with $\sum_{i=1}^{k}\alpha_{i}=1$  and $\sum_{i=1}^{k}\beta_{i}=1$, for $\mu_{i}, m_{i}\in M(X)$, then we have
$$D(\sum_{i=1}^{k}\alpha_{i}\mu_{i}, \sum_{i=1}^{k}\beta_{i}m_{i} )\leq \sum_{i=1}^{k}\alpha_{i}D(\mu_{i},m_{i})+\sum_{i=1}^{k}|\alpha_{i}-\beta_{i}|||m_{i}||,$$
where $||m||:=\sup_{0<||f||\leq 1}\frac{|\int f d\mu|}{||f||}.$
\end{prop}
\begin{proof}
\begin{align*}
\begin{split}
D(\sum_{i=1}^{k}\alpha_{i}\mu_{i}, \sum_{i=1}^{k}\beta_{i}m_{i} )&\leq D(\sum_{i=1}^{k}\alpha_{i}\mu_{i}, \sum_{i=1}^{k}\alpha_{i}m_{i})+ D(\sum_{i=1}^{k}\alpha_{i}m_{i},\sum_{i=1}^{k}\beta_{i}m_{i} )\\
&\leq
\sum_{i=1}^{k}\alpha_{i}D(\mu_{i},m_{i})+\sum_{i=1}^{k}|\alpha_{i}-\beta_{i}|||m_{i}||.
\end{split}
\end{align*}
\end{proof}
 To estimate the lower bound, we need the following $\mathbf{Horseshoe ~theorem} $ given by \cite{CliTho}.
 \begin{thm}\label{horse}\rm{\cite{CliTho}}
 Let $X$ be a shift space. Suppose that $\mathcal{G}\subset \mathcal{L}$ has $(W)$-specification and  $\mathcal{L}$ is edit approachable by $\mathcal{G}$. Then there exists an increasing sequence $\{X_{n}\}$ of compact $\sigma$-invariant subset of $X$ with the following properties.

 (1) Each $X_{n}$ is a topological transitive sofic shift.

 (2) There is $T\in \mathbb N$ such that for every $n$ and every $w\in \mathcal{L}(X_{n}),$
     there are $u, v\in \mathcal{L}$ with $|u|, |v|\leq n+T$ such that $uwv\in \mathcal{G}$.

 (3) Every invariant measure on $X$ is entropy approachable by ergodic measures on $X_{n}$: for any $\eta>0$, any $\mu\in \mathcal{M}_{\sigma}(X)$ , and any neighborhood $U$ of $\mu$ in $\mathcal{M}_{\sigma}(X), $ there exist $n\geq 1$ and $\mu^{'}\in \mathcal{M}_{\sigma}^{e}(X_{n})\cap U$ such that $h_\mu^{'}(\sigma)>h_\mu(\sigma)-\eta$ holds.
 \end{thm}

 \begin{lem}\rm{\cite{PfiSul}}\label{lem3.1}
 Let $K\subset M_{\sigma}(X)$ be a compact connected non-empty set. Then there exists a sequence $\{\alpha_{1},\alpha_{2},\cdots\}$ of $K$, such that
 $$\overline{\{\alpha_{j}:j\in\mathbb N, j>n\}}=K, \forall n\in \mathbb N, \text{and}~ \lim\limits_{j\to\infty}d(\alpha_{j},\alpha_{j+1})=0.$$
 \end{lem}
 \begin{defn}
 Given $\mu\in M_\sigma(X)$ and $\epsilon>0$, let
 $$\mathcal{L}^{\mu, \epsilon}_{n}:=\{w\in \mathcal{L}_{n}:~D(\mathcal{E}_{n}(x),\mu)<\epsilon~\text{for all}~x\in [w]\}.$$
 \end{defn}

  By Proposition 1.1 in \rm{\cite{PfiSul}}, we have the following result.
\begin{lem} \label{lem3.2}
For any ergodic measure $\mu\in M_{\sigma}^{e}(X)$ and $\epsilon>0$,$\delta>0,$ there exists $N(\mu,\epsilon,\delta)\in \mathbb N$ such that for $n\geq N(\mu,\epsilon,\delta)$, we have
$\sharp\mathcal{L}_{n}^{\mu,\epsilon }\geq e^{n(h_{\mu}(\sigma)-\delta)}.$
\end{lem}
 Now we begin to show the lower bound of Theorem \ref{main}.
\subsection{Choose the sequence $\{n_{j}\}_{j\geq 1}$}
 Let $\eta>0$, and $h^{*}=\inf\{h_{\mu}(\sigma)+\int\varphi d\mu:\mu\in K\}-\eta.$ We only need to prove
  $P_{G_{K}(X,\sigma)}(\varphi) \geq h^{*}.$
We choose $\epsilon_{0}  $  small enough, such that for any $\mu\in B(\alpha , \epsilon_{0})$, we have

\begin{align}\label{eq1}
|\int \varphi d\mu-\int \varphi d\alpha |<\frac{\eta}{2}.
\end{align}
 By Horseshoe Theorem \ref{horse}, we choose a sequence of measures $\{\alpha_{j}\}_{j\geq 1}$ in $K$ satisfying Lemma \ref{lem3.1}.
Then for any $j$,  there exists $X_{j}\subset X$, $\mu_{j}\in M_{\sigma}^{e}(X_{j}),$ and $ \epsilon_{j}<\epsilon_{0}$
 such that

 $$\mu_{j}\in B(\alpha_{j}, \frac{\epsilon_{j}}{2}), ~~\text{and}~~h_{\mu_{j}}(\sigma)>h_{\alpha_{j}}(\sigma)-\frac{\eta}{2}.$$
It follows from Lemma \ref{lem3.2} that there exists $\hat{N}_{j}$ such that $n>\hat{N}_{j}$ satisfies
 $$\sharp\mathcal{L}_{n}^{\mu_{j},\frac{\epsilon_{j}}{2}} \geq e^{n(h_{\mu_{j}}(\sigma)-\frac{\eta}{2})}.$$
Thus, for any $n>\hat{N}_{j}$,

$$\sharp\mathcal{L}_{n}^{\alpha_{j},\epsilon_{j}}\geq \sharp\mathcal{L}_{n}^{\mu_{j},\frac{\epsilon_{j}}{2}} \geq e^{n(h_{\alpha_{j}}(\sigma)-\eta)}.$$
 By the assumptions and Proposition \ref{prop2.2}, $\mathcal{L}$ is edit approachable by some $\mathcal{F}$ which has free concatenation property. We can define a map $\phi_{\mathcal{F}}:\mathcal{L}\to \mathcal{F}$ such that $\hat{d}(w,\phi_{\mathcal{F}}(w))\leq g(|w|).$  And then we can define a map $\Phi:\mathcal{L}^{*}\to \mathcal{F}$ by editing then gluing. That is
 $$(w^{1}, \cdots, w^{n})\mapsto \phi_{\mathcal{F}}(w^{1})\phi_{\mathcal{F}}(w^{2})\cdots \phi_{\mathcal{F}}(w^{n}),$$
 where $\mathcal{L}^{*}:=\{(w^{1}, \cdots, w^{n}):w^{j}\in \mathcal{L}, 1\leq j\leq n, n\in\mathbb N\}.$
By Lemma \ref{lem2.1}, (\ref{eq1}) and $\frac{g(n )}{n }\to 0$, we can choose $n_{j}\to\infty,$ such that  $n_{j}>\hat{N}_{j}$,
 \begin{align}\label{eq2}
 \frac{g(n_{j})}{n_{j}}<<\eta,
\end{align}
 \begin{align}\label{eq3}
 \left|\frac{S_{|\phi_{\mathcal{F}}(  w^{j})|}\varphi( y)}{|\phi_{\mathcal{F}}(  w^{j})|}-\frac{S_{n_{j} }\varphi(   x)}{n_{j}}\right|\leq \frac{\eta}{2},~~
\left|\frac{S_{n_{j} }\varphi(  x )}{n_{j}}-\int\varphi d\alpha_{j} \right|\leq \frac{\eta}{2},
 \end{align}
and
 \begin{align}\label{eq4}
\frac{g(n_{j})h_{top}(\sigma)}{n_{j}-g(n_{j})}\leq \eta,
\end{align}
for any $w^{j}\in \mathcal{L}_{n_{j}}^{\alpha_{j},\epsilon_{j}}, x\in [w^{j}], y\in [\phi_{\mathcal{F}}(  w^{j})]$.
 Moreover,
  by Lemma \ref{lem2.2}, we obtain
 \begin{align}\label{b1}
D(\mathcal{E}_{n_{j}}(x),\mathcal{E}_{|\phi_{\mathcal{F}}(w^{j})|}(y))\to 0.
 \end{align} for each $w^{j}\in \mathcal{L}_{n_{j}}^{\alpha_{j},\epsilon_{j}}$ and any $x\in [w^{j}]$ ,$y\in [\phi_{\mathcal{F}}(w^{j})]$.
Finally, by Proposition \ref{prop2.1} and (\ref{eq2}), for any $v \in \mathcal{F}$ with $|v|$ is large enough
\begin{align}\label{eq5}
\sharp\{w \in \mathcal{L}_{n_{j}}^{\alpha_{j},\epsilon_{j}}:\phi_{\mathcal{F}}(w )=v \}\leq e^{ |v |\eta  }.
\end{align}
\subsection{Construction of Moran set $H$}
For brevity of notation, we write $\mathcal{D}_{j}=\mathcal{L}_{n_{j}}^{\alpha_{j},\epsilon_{j}}.$
Moreover, we pick a strickly increasing sequence ${N_{k}}\to \infty$ with $N_{k}\in \mathbb N,$
\begin{align}\label{eq6}
\begin{split}
&\lim\limits_{k\to\infty}\dfrac{n_{k+1}+g(n_{k+1})}{\sum_{j=1}^{k}(n_{j}-g(n_{j}))N_{j}}=0, ~~~~~~~~~
 \lim\limits_{k\to\infty}\dfrac {\sum_{j=1}^{k}(n_{j}+g(n_{j}))N_{j}}{\sum_{j=1}^{k+1}(n_{j}-g(n_{j}))N_{j}}=0
\end{split}
\end{align}
We now define a new sequences $\{n_{j}^{'}\}, \{\alpha_{j}^{'}\}$ and $\{\mathcal{D}_{j}^{'}\}$ by setting for
$j=N_{1}+\cdots N_{k-1}+q$
with $1\leq q\leq N_{k},$
$$\epsilon_{j}^{'}:=\epsilon_{k}, n_{j}^{'}:=n_{k}, \alpha_{j}^{'}:=\alpha_{k},\mathcal{D}_{j}^{'}:=\mathcal{D}_{k}. $$
Consider the map
 $\Phi: \Pi_{j=1}^{\infty}\mathcal{D}_{j}^{'}\to X ,$ defined by editing then gluing.
More precisely, given $\mathbf w=\{w^{j}\}_{j=1}^{\infty}\in \Pi_{j=1}^{\infty}\mathcal{D}_{j}^{'}$, let $v^{j}=\phi_{\mathcal{F}}(w^{j})\in \mathcal{F}$ and $$\Phi(\mathbf w)=v^{1}v^{2}\cdots.$$
 Put  $$H:=\Phi(\Pi_{j=1}^{\infty}\mathcal{D}_{j}^{'}).$$
Next we will prove the fact that $H\subset G_{K}(X,\sigma).$
For any $\mathbf w=\{w^{j}\}_{j=1}^{\infty}\in  \Pi_{j=1}^{\infty}\mathcal{D}_{j}^{'} $, we define $l_{j}=l_{j}(w^{j})=|\phi_{\mathcal{F}}(w^{j})|$ for the length of the words associated to the index $j$. Clearly,
\begin{align}\label{eq7}
|l_{j}-n_{j}^{'}|\leq g(n_{j}^{'}).
\end{align}
 Following from the construction of $\alpha^{'}_{j}$, we have    $A(\alpha^{'}_{j})=A(\alpha_{j})=K$.
For any $j=N_{1}+\cdots N_{k-1}+q , 1\leq q \leq N_{k}.$
 we define $t_{j}=\sum_{i=1}^{j}l_{i}$.
 By (\ref{eq6}) and (\ref{eq7}), we obtain
 \begin{align*}
 \frac{l_{j} }{t_{j}}\leq \frac{n_{j}^{'}+g(n_{j}^{'})}{\sum_{i=1}^{j}l_{i}}\to 0.
 \end{align*}
Hence, $A(\mathcal{E}_{t_j} (  \Phi(\mathbf w)))=A(\mathcal{E}_{ n} (  \Phi(\mathbf w))) $.
Then it is sufficient to show that
$$\lim\limits_{j\to\infty}D(\mathcal{E}_{t_{j}} (  \Phi(\mathbf w)),\alpha_{ j }^{'})=0.$$
Assume that  $j=N_{1}+\cdots N_{k-1}+q , 1\leq q \leq N_{k},$  we have
$\alpha_{ j }^{'}=\alpha_{k}$. Define $ c_{k}:=N_{1}+N_{2}+\cdots +N_{k}$. Then we can make the following estimate,
\begin{align*}
\begin{split}
&D(\mathcal{E}_{t_{j}} (  \Phi(\mathbf w)),\alpha_{ j }^{'})\\&\leq\dfrac{t_{c_{k-2}}}{t_{j}}D(\mathcal{E}_{t_{ c_{k-2}}} (  \Phi(\mathbf w)),\alpha_{ j }^{'})+\dfrac{t_{c_{k-1}}-t_{c_{k-2}}}{t_{j}}D(\mathcal{E}_{t_{c_{k-1}}-t_{c_{k-2}}} ( \sigma^{t_{c_{k-2}}} \Phi(\mathbf w)),\alpha_{ j }^{'})\\
&+\dfrac{t_{j}-t_{c_{k-1}}}{t_{j}}D(\mathcal{E}_{t_{j}-t_{c_{k-1}}} ( \sigma^{t_{c_{k-1}}} \Phi(\mathbf w)),\alpha_{ j }^{'})\\
&\leq\dfrac{t_{c_{k-2}}}{t_{j}}D(\mathcal{E}_{t_{ c_{k-2}}} (  \Phi(\mathbf w)),\alpha_{ j }^{'})+\dfrac{t_{c_{k-1}}-t_{c_{k-2}}}{t_{j}}D(\mathcal{E}_{t_{c_{k-1}}-t_{c_{k-2}}} ( \sigma^{t_{c_{k-2}}} \Phi(\mathbf w)),\alpha_{ k-1})\\
&+\dfrac{t_{c_{k-1}}-t_{c_{k-2}}}{t_{j}}D(\alpha_{k-1},\alpha_{k})
+\dfrac{t_{j}-t_{c_{k-1}}}{t_{j}}D(\mathcal{E}_{t_{j}-t_{c_{k-1}}} ( \sigma^{t_{c_{k-1}}} \Phi(\mathbf w)),\alpha_{k}).\\
\end{split}
\end{align*}
From (\ref{eq6}), (\ref{eq7}) and Lemma \ref{lem3.1}, we obtain
\begin{align*}
\begin{split}
&\dfrac{t_{c_{k-2}}}{t_{j}}\leq \dfrac{N_{1}n_{1}+\cdots N_{k-2}n_{k-2} +N_{1}g(n_{1})+\cdots N_{k-2}g(n_{k-2})}{N_{1}n_{1}+\cdots N_{k-1}n_{k-1}+qn_{k}-N_{1}g(n_{1})-\cdots N_{k-1}g(n_{k-1})-qg(n_{k})}\to 0,\\
&D(\alpha_{k-1},\alpha_{k})\to 0.\\
\end{split}
 \end{align*}
For $1\leq i\leq N_{k-1}$,  taking any $x_{i}\in [w^{c_{k-2}+i}]$, we can make the following estimate
\begin{align*}
\begin{split}
&D(\mathcal{E}_{t_{c_{k-1}}-t_{c_{k-2}}} ( \sigma^{t_{c_{k-2}}} \Phi(\mathbf w)),\alpha_{ k-1})\\
&\leq D(\sum_{i=1}^{N_{k-1} }\frac{l_{c_{k-2}+i }}{t_{c_{k-1}}-t_{c_{k-2}}}\mathcal{E}_{l_{c_{k-2}+i }}  ( \sigma^{t_{c_{k-2}+i-1}} \Phi(\mathbf w)),\sum_{i=1}^{N_{k-1} }\frac{n_{k-1}}{N_{k-1}n_{k-1}}\mathcal{E}_{ n_{k-1}} ( x_{i} ))\\&+D(  \sum_{i=1}^{N_{k-1} }\frac{n_{k-1}}{N_{k-1}n_{k-1}}\mathcal{E}_{ n_{k-1}} ( x_{i} ),\alpha_{ k-1})\to 0, \\
\end{split}
\end{align*}
 where the above inequality follows from Proposition \ref{prop3.1}, (\ref{b1}) and (\ref{eq6}), (\ref{eq7}).
Likewise, choose any $y_{i}\in [w^{c_{k-1}+i+1}]$, for $0\leq i\leq q-1,$
\begin{align*}
\begin{split}
&D(\mathcal{E}_{t_{j}-t_{c_{k-1}}} ( \sigma^{t_{c_{k-1}}} \Phi(\mathbf w)),\alpha_{k})=D(\sum_{i=0}^{q-1}\frac{l_{ c_{k-1}+i+1 }}{t_{j}-t_{c_{k-1}}}\mathcal{E}_{l_{k}}(\sigma^{ t_{c_{k-1}+i}} \Phi(\mathbf w)),\alpha_{k})\\
&\leq D(\sum_{i=0}^{q-1}\frac{l_{ c_{k-1}+i+1 }}{t_{j}-t_{c_{k-1}}}\mathcal{E}_{l_{k}}(\sigma^{ t_{c_{k-1}+i}} \Phi(\mathbf w)),\frac{|w^{c_{k-1}+i+1}|}{\sum_{i=0}^{q-1}|w^{c_{k-1}+i+1}|}\mathcal{E}_{|w^{c_{k-1}+i+1}|}(y_{i}))\\
&+D(\frac{|w^{c_{k-1}+i+1}|}{\sum_{i=0}^{q-1}|w^{c_{k-1}+i+1}|}\mathcal{E}_{|w^{c_{k-1}+i+1}|}(y_{i}), \alpha_{k}) \to 0.
\end{split}
\end{align*}
Hence,
$$\lim\limits_{j\to\infty}D(\mathcal{E}_{t_{j}} (  \Phi(\mathbf w)),\alpha_{ j }^{'})=0.$$
So we proved $H\subset G_{K}(X,\sigma).$

\subsection{To estimate the lower bound}

Now we prove
$$P_{G_{K}(X,\sigma)}(\varphi) \geq h^{*}.$$
 From the choice of $N_{j}$ and the fact that
 $|l_{j}-n_{j}^{'}|\leq g(n_{j}^{'}),$ one can readily verify that  $\lim\limits_{j\to\infty}\frac{t_{j}}{t_{j+1}}=1. $
 For any $j\in \mathbb N$, we have
 \begin{align}\label{eq8}
\sharp \mathcal{D}_{j}^{'}=\sharp\mathcal{L}_{n_{j}^{'}}^{\alpha_{j}^{'},\epsilon_{j}^{'}}\geq e^{n_{j}^{'}(h_{\alpha^{'}_{j}}(\sigma)-\eta)}.
\end{align}
Let $\mathbf w=(w^{1},w^{2},\cdots)\in \Pi_{i=1}^{\infty}\mathcal{D}_{j}^{'}. $ Then for any $w^{j}$, by (\ref{eq3}), we have
\begin{align}\label{eq9}
\begin{split}
&\left|\int\varphi d\mathcal{E}_{l_{j}}(\sigma^{t_{j-1}}\Phi(\mathbf w))-\int\varphi d\alpha_{j}^{'}\right|\\&=\left|\frac{S_{l_{j}}\varphi(\sigma^{t_{j-1}}\Phi(\mathbf w))}{l_{j} }-\frac{S_{n_{j}^{'}}\varphi(  x)}{n^{'}_{j}}\right|+
\left|\frac{S_{n_{j}^{'}}\varphi(   x )}{n^{'}_{j}}-\int\varphi d\alpha_{j}^{'}\right|\\
&\leq \eta,
\end{split}
\end{align}
 where $j\geq 1, x\in [w^{j}]$.

Clearly, $H$ is a compact set. We can just consider finite words $\mathcal{C}$ of $H$ with the property that if $w\in \mathcal{C}$, then $[w]\cap H\neq \emptyset.$
For each $\mathcal{C} $ satisfying any $w\in \mathcal{C}$, $|w|>N$, we define the cover $\mathcal{C}^{'}$, in which each cylinder $w\in \mathcal{C}$ is replaced by its prefix $w|_{t_{j}}$ where $t_{j}\leq |w|<t_{j+1}.$
Then for any $\hat{h}<h^{*}-4\eta$,
\begin{align*}
\begin{split}
M(Z,\hat{h},\varphi, N)&=\inf\Big\{\sum_{[w_{0}w_{1}\cdots w_{m}]\in \mathcal{C}}\exp\Big(-\hat{h}(m+1)+\sup_{x\in[w_{0}w_{1}\cdots w_{m}] }\sum_{k=0}^{m}\varphi(\sigma^{k}x)\Big)\Big\}\\
&\geq \inf\Big\{\sum_{[w_{0}w_{1}\cdots w_{m}]\in \mathcal{C}^{'}}\exp\Big(-\hat{h} t_{j+1} +  \sup_{x\in[w_{0}w_{1}\cdots w_{m}] }\sum_{k=0}^{m}\varphi(\sigma^{k}x)\Big)\Big\}.
\end{split}
\end{align*}
Consider a specific $\mathcal{C}^{'}$ and let $s$ be the largest value of $j$ such that there exists $w|_{t_{j}}\in \mathcal{C}^{'}$.
 In the following, we set
 $$\mathcal{W}_{k}:=\Pi_{i=1}^{k}\Phi(\mathcal{D}_{j}^{'}),~~\overline{\mathcal{W}}_{s}:=\bigcup_{k=1}^{s} \mathcal{W}_{k}.$$
 For any $w^{j}\in \mathcal{D}_{j}^{'}$ and $l_{j}=l_{j}(w^{j})$, we can use (\ref{eq4}) and (\ref{eq8}) to get
 \begin{align}\label{eqc1}
 \sharp\mathcal{D}_{j}^{'}\geq e^{n_{j}^{'}(h_{\alpha_{j}^{'}}(\sigma)-\eta)}\geq  e^{l_{j}(h_{\alpha_{j}^{'}}(\sigma)-\eta)-g(n_{j}^{'})(h_{\alpha_{j}^{'}}(\sigma)-\eta)}
 \geq e^{l_{j}(h_{\alpha_{j}^{'}}(\sigma)-\eta)-l_{j}\eta}.
 \end{align}
 Furthermore, (\ref{eq5}),(\ref{eqc1}) and (\ref{eq9}) yields
 \begin{align}\label{eq10}
\begin{split}
&\sharp\Phi(\mathcal{D}_{j}^{'})\geq e^{l_{j}(h_{\alpha_{j}^{'}}(\sigma)-\eta)-2\hat{l}_{j}\eta} \\&\geq \exp\Big({l_{j}(h_{\alpha^{'}_{j}}(\sigma)+\int\varphi d\alpha_{j}^{'}-\eta)}-S_{l_{j}}\varphi(\sigma^{t_{j-1}}\Phi(\mathbf w)) -3\hat{l}_{j} \eta \Big)\\
&\geq \exp\Big({l_{j}h^{*}}-S_{l_{j}}\varphi(\sigma^{t_{j-1}}\Phi(\mathbf w))-3\hat{l}_{j}\eta\Big),
\end{split}
\end{align}
for any $\mathbf w=(w^{1},w^{2},\cdots)\in \Pi_{i=1}^{\infty}\mathcal{D}_{j}^{'}  $
 and  $l_{j}=l_{j}(w^{j})$, $\hat{l}_{j}:=\max_{w^{j}\in \mathcal{D}^{'}_{j}}l_{j}.$
Given $u:=\Phi(\mathbf w)\in \mathcal{W}_{k}$,
by (\ref{eq10}), we have
 \begin{align}\label{eq11}
\begin{split}
|\mathcal{W}_{k}|
 \geq \exp\Big( {t_{k}h^{*}}-3\hat{t}_{k}\eta-S_{t_{k}}\varphi( \Phi(\mathbf w)) \Big),
\end{split}
\end{align}
for any $\mathbf w=(w^{1},w^{2},\cdots)\in \Pi_{i=1}^{\infty}\mathcal{D}_{j}^{'}  $ and $\hat{t}_{k}:=\sum_{i=1}^{k}\hat{l}_{i}.$
 For $1\leq j\leq k$,  we say the word $v_{1}\cdots v_{j}\in \mathcal{W}_{j}$ is a prefix of $w=w_{1}\cdots w_{k}\in \mathcal{W}_{k}$ if $v_{i}=w_{i}, i=1,\cdots,j$. Note that each $w\in \mathcal{W}_{k}$ is the prefix of exactly $\dfrac{|\mathcal{W}_{s}|}{|\mathcal{W}_{k}|}$ words of $ \mathcal{W}_{s} $. If $\mathcal{W}\subset \overline{\mathcal{W}}_{s}$ contains a prefix of each word of $\mathcal{W}_{s},$ then
 $$\sum_{k=1}^{s}\dfrac{|\mathcal{W}\cap\mathcal{W}_{k}||\mathcal{W}_{ s}|}{|\mathcal{W}_{k}|}\geq |\mathcal{W}_{s}|.$$
 If $\mathcal{W}$ contains a prefix of each word of $\mathcal{W}_{s}$, we have
 $$\sum_{k=1}^{s}\dfrac{|\mathcal{W}\cap\mathcal{W}_{k}| }{|\mathcal{W}_{k}|}\geq 1.$$
It follows from (\ref{eq11}) that

 $$\sum_{\mathcal{C}^{'}} \exp\Big(-{t_{j}h^{*}}+S_{t_{j}}\varphi( \Phi(\mathbf w))+3\hat{t}_{j}\eta\Big)\geq 1.$$
By (\ref{eq6}), we have $\frac{t_{j}}{t_{j+1}}\to 1$ and $\frac{\hat{t}_{j}}{t_{j+1}}\to 1$, moreover, we ${\bf claim}$ that for any $\mathbf w=(w^{1},w^{2},\cdots)\in \Pi_{i=1}^{\infty}\mathcal{D}_{j}^{'}, $
\begin{align*}
 \begin{split}
 \Big( {t_{j}h^{*}}-S_{t_{j}}\varphi( \Phi(\mathbf w))-3\hat{t}_{j}\eta\Big)- \Big(\hat{h} t_{j+1} - \sup_{x\in[\Phi(\mathbf w)|_{(m+1)} ] } \sum_{k=0}^{m}\varphi(x)\Big)
 >0.
 \end{split}
 \end{align*}
Then for $N$ large enough,
$$M(Z,\hat{h},\varphi,  N)=\inf\Big\{\sum_{[w_{0}w_{1}\cdots w_{m}]\in \mathcal{C}}\exp\Big(-\hat{h}(m+1)+\sup_{x\in[w_{0}w_{1}\cdots w_{m}] }\sum_{k=0}^{m}\varphi(x)\Big)\Big\}\geq 1. $$
Moreover,
 $$M(Z,\hat{h},\varphi )\geq 1,$$
so we have $P_{H}(\varphi)\geq \hat{h}.$ Hence,
 $P_{G_{K}(X,\sigma)}(\varphi)\geq P_{H}(\varphi)\geq\hat{h}$. Together with $\hat{h}<h^{*}-4\eta $ and $\eta$ is arbitrary small, this completes the proof.

 \begin{cro}\label{cro}
Let $X$ be a shift space with $\mathcal{L}=\mathcal{L}(X)$. Suppose that $\mathcal{G}\subset \mathcal{L}$ has $(W)$-specification and  $\mathcal{L}$ is edit approachable by $\mathcal{G}$, then  for any $\mu\in M_{\sigma}(X)$, we have
 $G_{\mu}(X,\sigma)\neq \emptyset.$
\end{cro}

  \begin{prop}
Let $X$ be a shift space with $\mathcal{L}=\mathcal{L}(X)$. Suppose that $\mathcal{G}\subset \mathcal{L}$ has $(W)$-specification and  $\mathcal{L}$ is edit approachable by $\mathcal{G}$, then $\mathcal{L}_{\varphi}$ is non-empty bounded interval. Furthermore, $\mathcal{L}_{\varphi}=\{\int \varphi d\mu :\mu\in M_{\sigma}(X)\}.$
\end{prop}
\begin{proof}
Let $\mathcal{I}_{\varphi}:=\{\int \varphi d\mu: \mu\in M_{\sigma}(X)\}$. Since $M_{\sigma}(X)$ is compact, then $\mathcal{I}_{\varphi}$ is bounded. We first show $\mathcal{I}_{\varphi}=\mathcal{L}_{\varphi}.$ For any $\alpha\in \mathcal{I}_{\varphi} $, we can choose $\mu\in M_{\sigma}(X)$ such that $\alpha=\int \varphi d\mu$. By Corollary \ref{cro}, there exists $x\in G_{\mu}$, so $\alpha\in \mathcal{L}_{\varphi}$, this shows $\mathcal{I}_{\varphi}\subset \mathcal{L}_{\varphi}.$ On the other hand,
for any $\alpha\in \mathcal{L}_{\varphi}$, we can choose $x\in X(\varphi ,\alpha)$. Let $\mu$ be any weak$^*$ limit point of the sequence $\mathcal{E}_{n}(x)$. It is a standard result that $\mu\in M_{\sigma}(X)$, and it is easy to show that $\int \varphi d\mu =\alpha. $ So we have $\mathcal{I}_{\varphi}=\mathcal{L}_{\varphi}.$
Secondly, we show $\mathcal{I}_{\varphi}$ is an interval. By the convexity of $M_{\sigma}(X)$. Assume $\mathcal{I}_{\varphi}$ is not a single point. Let $\alpha_{1}, \alpha_{2}\in \mathcal{I}_{\varphi}.$ Let $\beta \in (\alpha_{1},\alpha_{2})$, and $\mu_{i}$ satisfying $\int \varphi d\mu_{i}=\alpha_{i}$ for $i=1,2.$ Let $t\in (0,1)$ satisfy $\beta =t\alpha_{1}+(1-t)\alpha_{2}$. we can get $m:=t\mu_{1}+(1-t)\mu_{2}$ satisfying $\int \varphi d\mu=\beta.$
\end{proof}
We give the following conditional variational principle.
 \begin{prop}\label{prop3.3}
 Let $X$ be a shift space with $\mathcal{L}=\mathcal{L}(X)$ and $\varphi :X\to \mathbb R $ be a continuous function. Suppose that $\mathcal{G}\subset \mathcal{L}$ has $(W)$-specification and  $\mathcal{L}$ is edit approachable by $\mathcal{G}$, then for any $\psi\in C(X)$,$\alpha\in \mathbb R$
$$P_{X(\psi, \alpha)}(\varphi)= \sup\{h_{\mu}(\sigma)+\int \varphi d\mu, \int \psi d\mu =\alpha\}.$$
\end{prop}
\begin{proof}
Let $F(\alpha):=\{\mu\in M_{\sigma}(X): \int\psi d\mu=\alpha\}$. $F(\alpha)$ is a closed set. The statement $\lim\limits_{n\to\infty}\frac{S_{n}\psi(x)}{n}=\alpha$ is equivalent to the statement $\mathcal{E}_{n}(x)$ has all its limit-points in $F(\alpha)$. Let
$$G^{F(\alpha)}:=\{x\in X: A(\mathcal{E}_{n}(x))=F(\alpha)\}.$$
For any $\mu\in F(\alpha),$ we have $G_{\mu}\subset G^{F(\alpha)}.$
So $h_{\mu}(\sigma)+\int\varphi d\mu =P_{G_{\mu}}(\varphi)\leq P_{G^{F(\alpha)}}(\varphi). $ And thus $\sup\{h_{\mu}(\sigma)+\int\varphi d\mu:\mu\in F(\alpha)\}\leq P_{G^{F(\alpha)}}(\varphi)$.
On the other hand, the upper bound can be verified by Theorem 3.1 in \cite{PeiChe}.

\end{proof}

Finally, we can show that the result about irregular set.

\noindent {\rm \bf Proof of Theorem 1.3}

Since the entropy map is upper semi-continuous, there exists ergodic measure $\mu\in M^{e}_{\sigma}(X)$ such that $P_{X}(\varphi)= h_{\mu}(\sigma)+\int\varphi d\mu  .$ By Lemma 2.1 in \cite{Thoo},   $\hat{X}(\psi )\neq \emptyset$ implies that there exists another  ergodic measure $\nu\in M^{e}_{\sigma}(X)$ and $\nu\neq \mu $. Let $p_{n}\in (0, 1)$ and $p_{n}\to 0$, we define $\nu_{n}:=p_{n}\nu+(1-p_{n})\mu$. Clearly, $\nu_{n}\to \mu$ and
 $$h_{\nu_{n}}(\sigma)+\int \varphi d\nu_{n}\to P_{X}(\varphi).$$
For any $\eta>0$, choose $N\geq 1, $ such that  for any $n\geq N$,
 $$h_{\nu_{n}}(\sigma)+\int \varphi d\nu_{n}\geq  P_{X}(\varphi)-\eta.$$
Furthermore, define the compact connected subset $K_{n}:=\{t\nu_{n}+(1-t)\mu:t\in [0,1]\}\subset M_{\sigma}(X).$
For any $n\geq 1$, we have
$$\hat{X}(\psi )  \supset \{x\in X: A(\mathcal{E}_{j}(x))=K_{n}\}.$$
It turns out that
\begin{align*}
\begin{split}
P_{X(\psi )}(\varphi)&\geq \inf_{m\in K_{n}}\{h_{m}(\sigma)+\int\varphi dm\}\\
&=\inf_{t\in [0,1]}\{th_{\nu_{n}}+(1-t)h_{\mu}(\sigma)+t\int\varphi d\nu_{n}+(1-t)\int \varphi d\mu\}
\\
&\geq \inf_{t\in [0,1]}\{t(P_{X}(\varphi)-\eta)+(1-t)P_{X}(\varphi)\}=P_{X}(\varphi)-\eta.
\end{split}
\end{align*}
Since $\eta$ can be arbitrary small, we complete the proof.

 \section{Applications}
 \subsection{Hausdorff dimension}
 In this section, we use the Hausdorff dimension to describe the level set. We use a metric defined by Gatzouras and Peres in \cite{DimPere}.
 Given a strictly positive continuous function $\varphi: X\to \mathbb R$, associate with it a metric $d_{\varphi}$ defined by for any $x, y\in X$, define

 \begin{equation*}
d_{\varphi}(x,y)=
\begin{cases}e^{-\sup_{z\in [x\wedge y]}S_{|x\wedge y|}\varphi(z)} ~~~~~if~ x\neq y \\[5pt] 0 ~~~~~~~~~~~~~~~~~~~~~~~~~~~~if~ x=y.\\
\end{cases}
\end{equation*}
Observe that $d_{\varphi}$ induces the product topology
on $X$, since the strict positivity of $\varphi$ implies that $S_{n}\varphi(x)\to \infty$ for all $x\in X$. Furthermore, we can use the metric $d_{\varphi}$ to define Hausdorff dimension denoted by ${\rm dim}_{\varphi}(\cdot)$. Together with the definition of topological pressure, we
 readily get that for any set $Z\subset X,$ the Hausdorff dimension of $Z$ is a unique solution of Bowen equation $P_{Z}(-s\varphi)=0,i.e., s={\rm dim}_{\varphi}(Z).$ Moreover, by Theorem \ref{main2} and Proposition \ref{prop3.3}, we have the following conditional variational principles.
 \begin{prop}\label{cor3.2}
Let $X$ be a shift space with $\mathcal{L}=\mathcal{L}(X)$. Suppose that $\mathcal{G}\subset \mathcal{L}$ has $(W)$-specification and  $\mathcal{L}$ is edit approachable by $\mathcal{G}$, then for any $\psi\in C(X)$,
$${\rm dim}_{\varphi}(X(\psi, \alpha), \varphi)= \sup\{ \frac{h_{\mu}(\sigma)}{\int \varphi d\mu} , \int \psi d\mu =\alpha\}.$$
\end{prop}

 \begin{prop}
 Let $X$ be a shift space with $\mathcal{L}=\mathcal{L}(X)$ and $\varphi :X\to \mathbb R $ be a continuous function. Suppose that $\mathcal{G}\subset \mathcal{L}$ has $(W)$-specification and  $\mathcal{L}$ is edit approachable by $\mathcal{G}$, then for any $\psi\in C(X)$,  either $\hat{X}(\psi )=\emptyset$, or
$${\rm dim}_{\varphi}(\hat{X}(\psi ))=  {\rm dim}_{\varphi}( X ).$$
 \end{prop}

\subsection{Subshifts}
 In this section, we study the level set of two main examples ($S$-gap shifts and $\beta$-shifts).

{\bf $S$-gap shifts} For a subshift of $\{0,1\}^{\mathbb N}$, fixed $S\subset \{0,1,2\cdots\}$, the number of $0$ between consecutive $s$ is an integer in $S$. That is, the language
  $$\{0^{n}10^{n_1}0^{n_2}0^{n_3}\cdots0^{n_k}1: 1\leq i\leq k ~\text{and}, n,m\in \mathbb N\},$$
  together with $\{0^{n}:n\in \mathbb N\}$, and this subshift is denoted by $\Sigma_{S}.$

   {\bf $\beta$-shifts}
  Fix $\beta>1$, write $b=\lceil \beta\rceil, $ and let $w^{\beta}\in \{0,1,\cdots,b-1\}^{\mathbb N}$ be the greedy $\beta$-expansion of $1$. Then $w^{\beta}$ satisfies $\sum_{j=1}^{\infty}w_{j}^{\beta} \beta^{-j}=1,$ and
  has the property that $\sigma^{j}(w^{\beta})\prec w^{\beta}$ for all $j\geq 1$, where $\prec$ denotes the lexicographic ordering. The $\beta$-shift is defined by
  $$\Sigma_{\beta}=\{x\in \{0,1,\cdots,b-1\}^{\mathbb N}:\sigma^{j}(x)\prec w^{\beta} ~\text{for all } j\geq 1\}.$$
In fact, in \cite{CliTho}, Climenhaga, Thompson and Yamamoto showed that all the factors of $S$-gap shifts and $\beta$-shifts satisfy the the conditions of Theorem\ref{main}(i.e., These subshifts have non-uniform structure). Hence,  for $X=\Sigma_{S}~\text{or}~\Sigma_{\beta}$, and $\varphi, \psi\in C(X), \alpha\in \mathbb R,$ we have
$${\rm dim}_{\varphi}(X(\psi, \alpha), \varphi)= \sup\{ \frac{h_{\mu}(\sigma)}{\int \varphi d\mu} , \int \psi d\mu =\alpha\},$$
and either $\hat{X}(\psi )=\emptyset$, or
$${\rm dim}_{\varphi}(\hat{X}(\psi ))=  {\rm dim}_{\varphi}( X ).$$
Accordingly, we also can use the topological pressure to describe these level set just like Theorem \ref{main} and Theorem \ref{main2}.


\begin{thebibliography}{50}



\bibitem {BarSauSch} L. Barreira, B. Saussol \& J. Schmeling, Higher-dimensional
multifractal analysis,  \emph{J. Math. Pures Appl}. {\bf 81} (2002),
67-91.

\bibitem{BarSch}  L. Barreira \& J. Schmeling, Sets of ¡°non-typical¡± points have full topological entropy and
full Hausdorff dimension, \emph{Israel J. Math.} {\bf 116} (2000), 29-70.

\bibitem {BarSau} L. Barreira \& B. Saussol, Variational principles
and mixed multifractal spectra,  \emph{Trans. Amer. Math. Soc}. {\bf
353} (2001), 3919-3944.


\bibitem{CheKupShu} E. Chen, T. Kupper \& L. Shu, Topological entropy for divergence points,
\emph{Ergodic Theory Dynam. Systems}  {\bf 25} (2005), 1173-1208.
\bibitem{CliTho}V. Climenhaga  \& D. Thompson \& K.Yamamoto,  Large deviations for systems with non-uniform structure. \emph{ Trans. Amer. Math. Soc. Ser.} to appear.
    \bibitem{DimPere}  G. Dimitrios \&  Y. Peres,
Invariant measures of full dimension for some expanding maps. \emph{
Ergodic Theory Dynam. Systems.} {\bf 17} (1997), 147-167.

\bibitem{Ols} L. Olsen, Multifractal analysis of divergence points of
deformed measure theoretical Birkhoff averages. IV: Divergence
points and packing dimension,  \emph{Bull. Sci. Math}. {\bf 132}
(2008), 650-678.


\bibitem {OlsWin} L. Olsen \& S. Winter, Multifractal analysis of divergence points of deformed
measure theoretical Birkhoff averages. II: Non-linearity, divergence
points and Banach space valued spectra, \emph{Bull. Sci. Math}.
 {\bf 131} (2007), 518-558.

\bibitem{Pes} Y. Pesin, Dimension Theory in Dynamical Systems, Contemporary Views and Applications,
Univ. of Chicago Press (1997).
\bibitem{PeiChe} Y. Pei \& E. Chen, On the variational principle for the
topological pressure for certain non-compact sets,  \emph{Sci. China. Math.
 } {\bf 53(4)} (2010),  1117-1128.


\bibitem{PfiSul} C.-E. Pfister \& W. Sullivan, On the topological
entropy of saturated sets,  \emph{Ergodic Theory Dynam. Systems}
{\bf 27} (2007), 929-956.

\bibitem{Thoo} D. Thompson, Irregular sets, the $\beta$-transformation and the almost specification property. \emph{Trans. Amer. Math. Soc. Ser.}   {\bf 364}(2012),5395-5414.
\bibitem{Tho2} D. Thompson, A variational principle for topological pressure for certain non-compact sets,  \emph{J. Lond. Math. Soc}. {\bf 80(3)} (2009), 585-602.
\bibitem{TakVer} F. Takens \& E. Verbitskiy, On the variational principle for the topological entropy of certain non-compact sets,
\emph{Ergodic Theory Dynam. Systems} {\bf 23(1)} (2003), 317-348.
 \bibitem{Wal}
P. Walter,  an introduction to ergodic theory, springer, Berlin,1982.
 \bibitem{ZhoChe} X. Zhou \& E. Chen, Multifractal analysis for the
historic set in topological dynamical systems, \emph{Nonlinearity}
{\bf26} (2013), 1975-1997.


\end{thebibliography}
\end{document}